\documentclass[a4paper]{article}

\usepackage{graphicx}
\usepackage{subcaption}
\usepackage[hidelinks]{hyperref}
\usepackage{geometry}
\usepackage{verbatim}
\usepackage{amsmath}
\usepackage{amssymb}
\usepackage{amsfonts}
\usepackage{amsthm}
\usepackage{verbatim}
\usepackage{fancyhdr}
\usepackage{algorithm}
\usepackage{algpseudocodex}
\usepackage{bigstrut}
\usepackage{authblk}
\usepackage[mathlines,pagewise,left]{lineno}

\theoremstyle{plain}
\newtheorem{theorem}{Theorem}

\newtheorem{lemma}{Lemma}

\theoremstyle{definition}
\newtheorem{remark}{Remark}

\newcommand{\blue}[1]{{\color{black}#1}}

\title {Gradient Descent Methods for Regularized Optimization}
\author[1]{Filip Nikolovski}
\author[2]{Irena Stojkovska}
\author[3]{Katerina Had\v{z}i-Velkova Saneva}
\author[3, 4]{Zoran Had\v{z}i-Velkov}

\affil[1]{\footnotesize Ss. Cyril and Methodius University in Skopje, Macedonia, Faculty of Mechanical Engineering, e-mail: {\tt filip.nikolovski@mf.ukim.edu.mk}}
\affil[2]{\footnotesize Ss. Cyril and Methodius University in Skopje, Macedonia, Faculty of Natural Sciences and Mathematics, Institute of Mathematics, e-mail: {\tt irenatra@pmf.ukim.mk}}
\affil[3]{\footnotesize Ss. Cyril and Methodius University in Skopje, Macedonia, Faculty of Electrical Engineering and Information Technologies, e-mail: {\tt saneva@feit.ukim.edu.mk, zoranhv@feit.ukim.edu.mk}}
\affil[4]{\footnotesize Macedonian Academy of Sciences and Arts, Skopje, Macedonia}

\begin{document}

\maketitle

\begin{abstract}\blue{
	Regularization is a widely recognized technique in mathematical optimization. It can be used to smooth out objective functions, refine the feasible solution set, or prevent overfitting in machine learning models. Due to its simplicity and robustness, the gradient descent (GD) method is one of the primary methods used for numerical optimization of differentiable objective functions. However, GD is not well-suited for solving $\ell^1$ regularized optimization problems since these problems are non-differentiable at zero, causing iteration updates to oscillate or fail to converge. Instead, a more effective version of GD, called the proximal gradient descent employs a technique known as soft-thresholding to shrink the iteration updates toward zero, thus enabling sparsity in the solution. Motivated by the widespread applications of proximal GD in sparse and low-rank recovery across various engineering disciplines, we provide an overview of the GD and proximal GD methods for solving regularized optimization problems. Furthermore, this paper proposes a novel algorithm for the proximal GD method that incorporates a variable step size. Unlike conventional proximal GD, which uses a fixed step size based on the global Lipschitz constant, our method estimates the Lipschitz constant locally at each iteration and uses its reciprocal as the step size. This eliminates the need for a global Lipschitz constant, which can be impractical to compute. Numerical experiments we performed on synthetic and real-data sets show notable performance improvement of the proposed method compared to the conventional proximal GD with constant step size, both in terms of number of iterations and in time requirements.
	}
\end{abstract}

\section{Introduction}

Regularization is a crucial technique in optimization and machine learning that prevents overfitting and improves the stability, convergence, and generalizability of optimization algorithms. Regularization techniques add a penalty to the loss function, discouraging the model from relying too heavily on any one feature or fitting the data too closely. 

In optimization problems that deal with fitting models to data (e.g. machine learning models), overfitting can occur when the optimization algorithm fits noise or irrelevant details in the data. Regularization prevents the optimizer from finding overly complex solutions that fit the training data too well, but generalize poorly to new, unseen data. Non-convex optimization problems often have many local minima, and optimization algorithms can get ``stuck'' at suboptimal solutions. Regularization can smooth out the optimization algorithms by introducing penalties that prevent extreme or irregular parameter values, which leads to better-behaved objective functions, which improve the convergence properties, thus avoiding sharp minima or erratic gradients that can trap optimization methods in local optima. In high-dimensional optimization problems, regularization can guide the search for the optimal solution by discouraging unnecessary exploration of irrelevant regions of the parameter space. 

$\ell^1$ (LASSO) regularization is suitable for optimization problems involving high-dimensional data since it forces many parameters to zero (i.e. induces sparsity), therefore selecting only the most important features or parameters of the input data, \cite{Tibshirani1996}. $\ell^2$ (Ridge or Tikhonov) regularization encourages smaller parameter values that stabilizes the solution to the optimization problems, \cite{Hoerl1970}. This regularization type, common in inverse problems, is used to stabilize ill-conditioned optimization problems by penalizing large solutions. In neural networks, dropout regularization technique randomly drops units during training, \blue{forcing the network not to be overly dependent on any one connection between neighboring layers}, thus improving robustness. A combination of Ridge and LASSO regularization, known as elastic-net, benefits from the both of them and gives better predictions, \cite{Zou2005}. Another regularization method that can do both shrinkage and selection (like LASSO), but is resistant to outliers or heavy-tailed errors, is LED-lasso regularization method, \cite{Wang2007}. $\ell^{\infty}$ (MASSO) regularization gives chance to build better models based on the information from all its predictors, when predictors have comparable influence, \cite{Dimovski2017}.

Gradient descent (GD) is the primary method for regularized optimization because it is computationally efficient, scalable, adaptable to different types of regularization, and can be applied to a wide variety of optimization problems, especially in machine learning. For convex optimization problems (such as those with convex loss functions and convex regularizers like $\ell^1$ and $\ell^2$), GD with an appropriate \blue{step size} converges to the global minimum of convex problems. One of the most important variants of GD is the proximal GD, which is more efficient than GD for non-differentiable regularizers, like the $\ell^1$ norm in LASSO regression. 

Our study is motivated by the fact that the $\ell^1$ regularization has found useful applications in sparse and low-rank recovery for many engineering disciplines \cite{app0}, \cite{app2}, \cite{app1}, such as wireless communications \cite{app3}, signal/image processing \cite{app1} and statistics \cite{app5}, \cite{app4}. Particular applications include compressive sensing, sparse regression and variable selection, sparse signals separation and sparse principal component analysis (PCA) \cite{app8}, \cite{app6}, \cite{app7}. 

In this paper, we present an overview of GD and proximal GD methods for solving regularized optimization problems, with a focus on their convergence properties under both fixed and variable \blue{step sizes}. We demonstrate that GD, when applied with an appropriately chosen fixed step size, achieves linear convergence for optimization problems under mild assumptions of differentiability and Lipschitz smoothness of the objective functions. Furthermore, for strongly convex functions, GD exhibits exponential (i.e., superlinear) convergence, reaching a desired \blue{calculation} precision in the optimal solution of convex optimization problems. When the objective function is a sum of differentiable and non-differentiable convex functions, as in the case of $\ell^1$ regularization, the proximal GD method becomes the preferred approach, preserving the convergence properties of conventional GD, albeit with an additional computational cost per iteration. Further, we introduce a variant of the proximal GD method with a variable step size. \blue{Numerical experiments performed on synthetic data sets, and one commonly used real-data set in machine learning show that the proposed method \blue{significantly} outperforms the proximal GD method with a constant step size.}

\section{Gradient Descent Method}\label{sec02}

Let $f : \mathbb{R}^d \to \mathbb{R}$ be a continuously differentiable and convex function. We are solving the unconstrained optimization problem:

\begin{equation}\label{UOProblem}
	\min_{x \in \mathbb{R}^d} f(x).
\end{equation}
We assume that there is a solution $x^*$ of \eqref{UOProblem}. One of the simplest approaches for solving this problem is the \textit{gradient descent method}, originally due to Cauchy \cite{Cauchy}. Starting from an initial approximation $x_0 \in \mathbb{R}^d$ of the solution, an iterative sequence $x_1, x_2, \ldots$ is constructed as:
\begin{equation}\label{GDMethod}
	x_{k+1} = x_k - \lambda_k \nabla f(x_k), \quad k = 0, 1, 2, \ldots,
\end{equation}
where $\nabla f(\cdot)$ denotes the gradient of $f$, $\lambda_k > 0$ are scalars called \textit{step sizes}, while the antigradient $-\nabla f(x_k)$ is called \textit{search direction}. Under suitable assumptions on $f$ it can be shown that the iterative sequence converges to a local minimum; see \cite[Theorem 1.2.4]{Nesterov}. The step sizes can be selected in various ways. The simplest choice is to set $\lambda_k = \lambda \equiv \mathrm{const.}$ for all $k$; more advanced techniques select $\lambda_k$ as a minimizer of $\psi(\lambda) = f(x_k - \lambda \nabla f(x_k))$ in the antigradient direction; see \cite[sec. 3.1]{Nocedal}.

\subsection{Convergence of Lipschitz smooth functions}

Let us assume that $f$ has a Lipschitz continuous gradient, that is:
\begin{equation}\label{LsmoothDef}
\| \nabla f(x) - \nabla f(y) \| \leq L \|x - y\|,
\end{equation}
for some $L > 0$ and all $x, y \in \mathbb{R}^d$ (here $\| \cdot \|$ denotes the Euclidean $2$-norm). We call such functions \textit{$L$-smooth}. We are interested in finding a constant step size $\lambda$ in the GD method \eqref{GDMethod} which leads to maximum decrease of the function value in successive iterations. For a continuously differentiable $L$-smooth function $f$ we have (see, e.g. \cite[pg.~25]{Nesterov}):
\begin{equation}\label{SmoothnessDef}
	f(y) \leq f(x) + \nabla f(x)^T (y - x) + \frac{L}{2} \|y - x\|^2,
\end{equation}
for all $x, y \in \mathbb{R}^d$. Putting $x = x_k$ and $y = x_{k+1} = x_k - \lambda \nabla f(x_k)$ we get:
\[
f(x_{k+1}) \leq f(x_k) + \left( \frac{L}{2}\lambda^2 - \lambda \right) \|\nabla f(x_k)\|^2, \quad k = 0, 1, 2, \ldots.
\]
Minimizing the function $\varphi(\lambda) = \frac{L}{2}\lambda^2 - \lambda$ gives us a minimum $\lambda^* = \frac{1}{L}$ and $\varphi(\lambda^*) = -\frac{1}{2L}$. Hence, when using a constant step in the GD method \eqref{GDMethod}, the maximum decrease is achieved for $\lambda_k = \lambda = \frac{1}{L}$, in which case:
\begin{equation}\label{GDMaxDecrese}
	f(x_{k+1}) \leq f(x_k) - \frac{1}{2L} \|\nabla f(x_k)\|^2, \quad k = 0, 1, 2, \ldots.
\end{equation}
We can now prove the following convergence result, \cite{Gartner2023}. 

\begin{theorem}\label{sec02:thm01}
	Let $f$ be continuously differentiable, convex and $L$-smooth function with a global minimizer $x^*$. Let $\{x_k\}$ be the sequence generated by the GD method \eqref{GDMethod} with a constant step size $\lambda_k \equiv \lambda = \frac{1}{L}$. Then:
	\[
	f(x_n) - f(x^*) \leq \frac{L}{2n} \|x_0 - x^*\|^2,\quad \text{for any } n \geq 1.
	\]
\end{theorem}

\begin{proof}
	Rearranging \eqref{GDMaxDecrese} and summing for $k=0, \ldots, n-1$ we obtain:
	\begin{equation}\label{sec02:eq00}
		\frac{1}{2L} \sum_{k=0}^{n-1} \|\nabla f(x_k)\|^2 \leq \sum_{k=0}^{n-1} \big( f(x_k) - f(x_{k+1}) \big) = f(x_0) - f(x_n).
	\end{equation}
	From the convexity of $f$, for any $x, y$ we have:
	\[
	f(y) \geq f(x) + \nabla f(x)^T (y - x) \Longleftrightarrow f(x) - f(y) \leq \nabla f(x)^T (x - y),
	\]
	and setting $x = x_k$ and $y = x^*$ gives:
	\begin{equation}\label{sec02:eq01}
		f(x_k) - f(x^*) \leq \nabla f(x_k)^T (x_k - x^*).
	\end{equation}
	Now, from \eqref{GDMethod} it follows that
	$\nabla f(x_k) = \frac{1}{\lambda} (x_k - x_{k+1})$, from where:
	\begin{align}
		\nabla f(x_k)^T (x_k - x^*) &= \frac{1}{\lambda}(x_k - x_{k+1})^T (x_k - x^*) = \notag \\
		&= \frac{1}{2\lambda} \left( \|x_k - x_{k+1}\|^2 + \|x_k - x^*\|^2 - \|x_{k+1} - x^*\|^2 \right) = \notag \\
		&= \frac{\lambda}{2} \|\nabla f(x_k)\|^2 + \frac{1}{2\lambda} \left( \|x_k - x^*\|^2 - \|x_{k+1} - x^*\|^2 \right), \label{sec02:VanillaAnalysis}
	\end{align}
	where in the second equality we used the following property of the norm: $\|u - v\|^2 = \|u\|^2 + \|v\|^2 - 2u^Tv$.	Summing up for $k = 0, 1, \ldots, n-1$ gives:
	\begin{align}\label{sec02:eq02}
		\sum_{k=0}^{n-1} \nabla f(x_k)^T (x_k - x^*) &= \frac{\lambda}{2}\sum_{k=0}^{n-1} \|\nabla f(x_k)\|^2 + \frac{1}{2\lambda} \left( \|x_0 - x^*\|^2 - \|x_n - x^*\|^2 \right) \leq \notag \\\
		&\leq \frac{\lambda}{2}\sum_{k=0}^{n-1} \|\nabla f(x_k)\|^2 + \frac{1}{2\lambda} \|x_0 - x^*\|^2.
	\end{align}
	Summing up \eqref{sec02:eq01}, and using \eqref{sec02:eq02} and \eqref{sec02:eq00} with $\lambda = \frac{1}{L}$ we obtain:
	\begin{align*}
		\sum_{k=0}^{n-1} \big( f(x_k) - f(x^*) \big) &\leq \sum_{k=0}^{n-1} \nabla f(x_k)^T (x_k - x^*) \leq \frac{\lambda}{2}\sum_{k=0}^{n-1} \|\nabla f(x_k)\|^2 + \frac{1}{2\lambda} \|x_0 - x^*\|^2 \\
		& \leq f(x_0) - f(x_n) + \frac{1}{2\lambda} \|x_0 - x^*\|^2,
	\end{align*}
	or equivalently:
	\[
		\sum_{k=1}^{n} \big( f(x_k) - f(x^*) \big) \leq \frac{1}{2\lambda} \|x_0 - x^*\|^2.
	\]
	From \eqref{GDMaxDecrese} we have $f(x_{k+1}) \leq f(x_k)$, so finally we get:
	\[
		f(x_n) - f(x^*) = \frac{1}{n}\sum_{k=1}^n \big( f(x_n) - f(x^*) \big) \leq \frac{1}{n}\sum_{k=1}^n \big( f(x_k) - f(x^*) \big) \leq \frac{1}{2n\lambda} \|x_0 - x^*\|^2,
	\]
	which gives the result by applying $\lambda = \frac{1}{L}$.
\end{proof}

It can be concluded, based on Theorem~\ref{sec02:thm01}, that the GD method \eqref{GDMethod} achieves $\varepsilon$-optimal solution in at most $\lceil L\|x_0 - x^*\|^2/(2\varepsilon) \rceil$ iterations, i.e. the complexity of the method under the assumptions of the theorem is $\mathcal{O}(1/\varepsilon)$.

\subsection{Convergence of $\mu$-strongly convex functions}

Assuming additional properties for the function $f$ may result in better convergence properties of the GD method \eqref{GDMethod}. Let us assume that $f$ is \textit{$\mu$-strongly convex}, i.e. there exists a parameter $\mu > 0$ such that for any $x, y \in\mathbb{R}^d$,
\begin{equation}\label{StrongConvexityDef}
	f(y) \geq f(x) + \nabla f(x)^T (y - x) + \frac{\mu}{2} \|y - x\|^2.
\end{equation}
Then we have the following convergence result, \cite{Gartner2023}.

\begin{theorem}\label{sec02:thm02}
	Let $f$ be continuously differentiable, $L$-smooth, and $\mu$-strongly convex function with a global minimizer $x^*$. Let $\{x_k\}$ be the sequence generated by the GD method \eqref{GDMethod} with a constant step size $\lambda_k \equiv \lambda = \frac{1}{L}$. Then:
	\begin{itemize}
		\item[(i)] Consecutive squared distances to $x^*$ decrease geometrically for any $k \geq 0$:
		\[
			\|x_{k+1} - x^*\|^2 \leq \left(1 - \frac{\mu}{L}\right) \|x_k - x^*\|^2,
		\]
		from where, for $n>0$:
		\[
			\|x_{n} - x^*\| \leq \left(1 - \frac{\mu}{L}\right)^{n/2} \|x_0 - x^*\|.
		\]
		\item[(ii)] The approximation error of the optimal value after $n>0$ iterations decreases exponentially in $n$:
		\[
			f(x_n) - f(x^*) \leq \frac{L}{2} \left(1 - \frac{\mu}{L}\right)^n \|x_0 - x^*\|^2.
		\]
	\end{itemize}
\end{theorem}

\begin{proof}
	Since $f$ is strongly convex, we obtain the following bound from \eqref{StrongConvexityDef} by putting $x = x_k$ and $y = x^*$:
	\begin{equation*}
		\nabla f(x_k)^T (x_k - x^*) \geq f(x_k) - f(x^*) + \frac{\mu}{2}\| x_k - x^*\|^2.
	\end{equation*}
	Combining this inequality and \eqref{sec02:VanillaAnalysis} we get:
	\begin{align}
		f(x_k) - f(x^*) &\leq \frac{\lambda}{2} \|\nabla f(x_k)\|^2 + \frac{1}{2\lambda} \left( \|x_k - x^*\|^2 - \|x_{k+1} - x^*\|^2 \right) - \frac{\mu}{2}\| x_k - x^*\|^2 \notag \\
		\Longleftrightarrow
		\|x_{k+1} - x^* \|^2 &\leq 2\lambda(f(x^*) - f(x_k)) + \lambda^2 \|\nabla f(x_k)\|^2 + (1 - \lambda \mu)\|x_k - x^*\|^2. \label{sec02:eq03}
	\end{align}
	Now, using \eqref{GDMaxDecrese} we have:
	\begin{equation*}
		f(x^*) - f(x_k) \leq f(x_{k+1}) - f(x_k) \leq -\frac{1}{2L} \|\nabla f(x_k)\|^2 = -\frac{\lambda}{2} \|\nabla f(x_k)\|^2,
	\end{equation*}
	or, equivalently:
	\begin{equation*}
		2\lambda(f(x^*) - f(x_k)) + \lambda^2 \|\nabla f(x_k)\|^2 \leq 0.
	\end{equation*}
	Thus, the last inequality and \eqref{sec02:eq03} yield:
	\[
		\|x_{k+1} - x^* \|^2 \leq (1 - \lambda \mu) \|x_k - x^*\|^2 = \left(1 - \frac{\mu}{L}\right) \|x_k - x^*\|^2,
	\]
	which proves the first part of (i). Applying this bound recursively, we get:
	\[
		\|x_{k+1} - x^* \|^2 \leq \left(1 - \frac{\mu}{L}\right)^2 \|x_{k-1} - x^*\|^2 \leq \ldots \leq \left(1 - \frac{\mu}{L}\right)^{k+1} \|x_0 - x^*\|^2,
	\]
	from where the second part of (i) follows for $k = n-1$. 
	
	Now, from \eqref{SmoothnessDef}, the smoothness definition of $f$, and the fact that $\nabla f(x^*) = 0$ we obtain:
	\begin{equation}\label{sec02:eq04}
		f(x_n) - f(x^*) \leq \nabla f(x^*)^T (x_n - x^*) + \frac{L}{2} \|x_n - x^*\|^2 = \frac{L}{2} \|x_n - x^*\|^2.
	\end{equation}
	Combining \eqref{sec02:eq04} and the result (i) yields:
	\[
		f(x_n) - f(x^*) \leq \frac{L}{2} \|x_n - x^*\|^2 \leq \frac{L}{2} \left(1 - \frac{\mu}{L}\right)^n \|x_0 - x^*\|^2,
	\]
	which proves (ii).
\end{proof}

Taking into account that $\ln{(x+1)} \leq x$ for $x > -1$, we  get that $\ln{(1 - \mu/L)} \leq -\mu/L$, so Theorem~\ref{sec02:thm02} implies that the GD method \eqref{GDMethod} achieves $\varepsilon$-optimal solution in at most $\left\lceil \frac{L}{\mu} \ln \left(\frac{L\|x_0-x^*\|^2}{2\varepsilon}\right) \right\rceil$ iterations, i.e. the complexity of the method under the assumptions of the theorem is $\mathcal{O}\big(\ln{(1/\varepsilon)}\big)$.

\section{Proximal Gradient Descent Method}\label{sec03}

As discussed in Section~\ref{sec02}, the GD method is a simple and efficient iterative tool for finding approximate solutions of the unconstrained optimization problem \eqref{UOProblem}. Now we turn our attention to a special form of the problem.

Let $F(x) \equiv f(x) + g(x)$ be a composite function where $f: \mathbb{R}^d \to \mathbb{R}$ is continuously differentiable, convex, and $L$-smooth, and $g: \mathbb{R}^d \to \mathbb{R}$ is continuous, convex, and possibly non-smooth. We are solving the problem:
\begin{equation}\label{ProxUOProblem}
	\min_{x \in \mathbb{R}^d} \left\{ F(x) \equiv f(x) + g(x) \right\}.
\end{equation}
We assume that there is a solution $x^*$ of \eqref{ProxUOProblem}. To devise a modification of the GD method that would be applicable in this case, we start by noting that the GD iterations \eqref{GDMethod} for solving problem \eqref{UOProblem} which employ a constant step size $\lambda_k \equiv \lambda$ can be seen as solutions of the sequence of subproblems:
\begin{equation}\label{sec03:eq00}
	x_{k + 1} = \mathop{\mathrm{argmin}}_{y \in \mathbb{R}^d} \left\{ f(x_k) + \nabla f(x_k)^T (y - x_k) + \frac{1}{2\lambda} \|y - x_k \|^2 \right\}.
\end{equation}
Since we can rewrite:
\begin{align*}
	f(x_k) &+ \nabla f(x_k)^T (y - x_k) + \frac{1}{2\lambda} \|y - x_k \|^2 = \\
	&= \frac{1}{2\lambda} \left( 2\lambda \nabla f(x_k)^T (y-x_k) + \|y-x_k\|^2 + \lambda^2 \|\nabla f(x_k)\|^2 - \lambda^2 \|\nabla f(x_k)\|^2 \right) + f(x_k) = \\
	&= \frac{1}{2\lambda} \| (y-x_k) + \lambda \nabla f(x_k) \|^2 - \frac{\lambda}{2} \|\nabla f(x_k)\|^2 + f(x_k) = \\
	&= \frac{1}{2\lambda} \left\| y - \big(x_k - \lambda \nabla f(x_k) \big) \right\|^2 - \frac{\lambda}{2} \|\nabla f(x_k)\|^2 + f(x_k),
\end{align*}
it follows that \eqref{sec03:eq00} is equivalent with the minimization problem:
\begin{equation}\label{sec03:eq01}
	x_{k+1} = \mathop{\mathrm{argmin}}_{y \in \mathbb{R}^d} \left\{ \frac{1}{2\lambda} \left\| y - \big(x_k - \lambda \nabla f(x_k) \big) \right\|^2 \right\}.
\end{equation}
Now, in light of problem \eqref{ProxUOProblem} we can modify \eqref{sec03:eq01} to include $g$:
\begin{align}
	x_{k + 1} 
	&= \mathop{\mathrm{argmin}}_{y \in \mathbb{R}^d} \left\{ \frac{1}{2\lambda} \left\| y - \big(x_k - \lambda \nabla f(x_k) \big) \right\|^2 + g(y)\right\}. \label{sec03:eq02}
\end{align}
For the purpose of solving \eqref{ProxUOProblem} we define a \textit{proximal mapping} as:
\begin{equation}\label{sec03:ProximalMapping}
	\mathrm{prox}_{\lambda} (z) = \mathop{\mathrm{argmin}}_{y \in \mathbb{R}^d} \left\{ \frac{1}{2\lambda} \left\| y - z \right\|^2 + g(y)\right\}.
\end{equation}
Then we can formulate a generalization of the GD method \eqref{GDMethod} which we call \textit{proximal gradient descent method} for solving \eqref{ProxUOProblem} given as follows: starting from an initial approximation $x_0 \in \mathbb{R}^d$, construct an iterative sequence $x_1, x_2, \ldots$ by setting:
\begin{equation}\label{PGDMethod}
	x_{k+1} = \mathrm{prox}_{\lambda} \big(x_k - \lambda \nabla f(x_k) \big), \quad k = 0, 1, 2, \ldots.
\end{equation}
Further, if we require the proximal GD method to have a form similar to the GD iterations \eqref{GDMethod}, then we need a quantity $G_{\lambda}(\cdot)$ such that $x_{k+1} = x_k - \lambda G_{\lambda}(x_k)$ for all $k \geq 0$, from where we can deduce that:
\begin{equation}\label{Glambda}
G_{\lambda} (x_k) =\frac{1}{\lambda} (x_k - x_{k+1}) = \frac{1}{\lambda} \left( x_k - \mathrm{prox}_{\lambda} \big(x_k - \lambda \nabla f(x_k) \big) \right).
\end{equation}
We call the quantity $G_\lambda$ \textit{generalized gradient of $f$} since it coincides with $\nabla f$ when $g \equiv 0$; see \cite[pg. 273]{Beck} for a detailed discussion.

The proximal GD method can also be interpreted as a fixed-point iteration. This follows from the fact that any solution $x^*$ of \eqref{ProxUOProblem} satisfies the optimality condition $0 \in \nabla f(x^*) + \partial g(x^*)$, which can be shown to be equivalent to:
\begin{equation}\label{fixedpoint}
x^* = \mathrm{prox}_\lambda \big(x^* - \lambda \nabla f(x^*) \big),
\end{equation}
for some $\lambda > 0$; see \cite[pg. 150]{ParikhBoyd}. Here, $\partial g(x) \subset \mathbb{R}^d$ is the \textit{subdifferential} of $g$ at $x$, defined by $ \partial g(x) = \{y\in\mathbb{R}^d \; | \; g(z) \geq g(x) + y^T(z-x) \; \text{for all} \; z \in \mathbb{R}^d \}.$ This allows us to construct a working algorithm implementation for the proximal GD method. The full description is given in Algorithm~\ref{alg:01}.

\begin{algorithm}[t]
	\caption{Proximal Gradient Descent Method with Constant Step Size}
	\begin{algorithmic}[1]
		\State \textbf{input:} $x_0 \in \mathbb{R}^d$, $N \in \mathbb{N}$
		\State \textbf{calculate:} $L$ and set $\lambda = \frac{1}{L}$
		\State \textbf{set:} $k=0$
		\While{$k < N$}
		\State $x_{k+1} = \mathrm{prox}_\lambda \big(x_k - \lambda \nabla f(x_k) \big)$
		\State $k \leftarrow k + 1$
		\EndWhile
	\end{algorithmic}
	\label{alg:01}
\end{algorithm}

\subsection{Convergence properties when $f$ is Lipschitz smooth}
To prove a convergence result for the proximal GD method (Algorithm~\ref{alg:01}) similar to the one about the GD method given in Theorem \ref{sec02:thm01}, first we prove the following lemma, \cite{Sarkar2015}.

\begin{lemma}\label{sec03:lem01}
	Let $F(x) = f(x) + g(x)$ with $f : \mathbb{R}^d \to \mathbb{R}$ continuously differentiable, convex and $L$-smooth, and $g : \mathbb{R}^d \to \mathbb{R}$ continuous, convex and possibly non-smooth. Let $0 < \lambda \leq \frac{1}{L}$ and $\{x_k\}$ be the sequence generated by Algorithm~\ref{alg:01}. Then for any $z \in \mathbb{R}^d$ and $k \geq 0$:
	\begin{equation*}
		F(x_{k+1}) \leq F(z) - G_\lambda (x_k)^T (z - x_k) - \frac{1}{2\lambda} \|x_{k+1} - x_k \|^2.
	\end{equation*}
\end{lemma}

\begin{proof}
	Let $z \in \mathbb{R}^d$ and $k \geq 0$ be arbitrary. From the convexity of $f$ we have $f(z) \geq f(x_k) + \nabla f(x_k)^T (z - x_k) \Longleftrightarrow f(x_k) \leq f(z) - \nabla f(x_k)^T(z - x_k)$. Combining this inequality with the smoothness of $f$ characterized by \eqref{SmoothnessDef} with $y = x_{k+1}$ and $x = x_k$ we get:
	\begin{align}
		f(x_{k+1}) &\leq f(x_k) + \nabla f(x_k)^T (x_{k+1} - x_k) + \frac{L}{2} \|x_{k+1} - x_k \|^2 \leq \notag\\
		& \leq f(z) - \nabla f(x_k)^T (z - x_{k+1}) + \frac{L}{2} \|x_{k+1} - x_k \|^2. \label{sec03:eq04}
	\end{align}
	Similarly, from the convexity of $g$ we obtain a bound:
	\begin{equation}\label{sec03:eq05}
		g(x_{k+1}) \leq g(z) - \gamma_{k+1}^T (z - x_{k+1}),
	\end{equation}
	for an arbitrary $\gamma_{k+1} \in \partial g(x_{k+1})$. Now taking into consideration \eqref{sec03:eq02}, \eqref{sec03:ProximalMapping} and \eqref{PGDMethod}:
	\begin{align*}
		x_{k+1} &= \mathrm{prox}_\lambda \big( x_k - \lambda \nabla f(x_k) \big) = \mathop{\mathrm{argmin}}_{y \in \mathbb{R}^d} \left\{ g(y) + \frac{1}{2\lambda} \| y - (x_k - \lambda \nabla f(x_k)) \|^2 \right\} \\
		& \Longrightarrow 0 \in \partial g(x_{k+1}) + \frac{1}{\lambda} \big( x_{k+1} - (x_k - \lambda \nabla f(x_k)) \big) \\
		&\Longleftrightarrow \frac{1}{\lambda} (x_k - x_{k+1}) - \nabla f(x_k) \in \partial g(x_{k+1}) \\
		&\Longleftrightarrow G_\lambda (x_k) - \nabla f(x_k) \in \partial g(x_{k+1}).
	\end{align*}
	By the last statement, if we sum \eqref{sec03:eq04} and \eqref{sec03:eq05}, using $\gamma_{k+1} = G_\lambda(x_k) - \nabla f(x_k)$ we obtain:
	\begin{align*}
		f(x_{k+1}) + g(x_{k+1}) &\leq f(z) + g(z) - G_\lambda(x_k)^T (z - x_{k+1}) + \frac{L}{2} \|x_{k+1} - x_k \|^2 \\
		\Longleftrightarrow F(x_{k+1}) &\leq F(z) - G_\lambda(x_k)^T (z - x_{k+1}) + \frac{L}{2} \|x_{k+1} - x_k \|^2.
	\end{align*}
	Since $x_k - x_{k+1} = \lambda G_\lambda(x_k) \Longleftrightarrow \frac{1}{\lambda} \|x_{k+1} - x_k\|^2 = \lambda \|G_\lambda (x_k)\|^2$ and $\lambda \leq \frac{1}{L}$, we obtain:
	\begin{align*}
		F(x_{k+1}) &\leq F(z) - G_\lambda(x_k)^T (z - x_k + \lambda G_\lambda(x_k)) + \frac{L}{2} \|x_{k+1} - x_k \|^2 = \\
		&= F(z) - G_\lambda(x_k)^T (z - x_k) - \lambda \|G_\lambda(x_k)\|^2 + \frac{L}{2} \|x_{k+1} - x_k \|^2 \leq \\
		&\leq F(z) - G_\lambda(x_k)^T (z - x_k) - \frac{1}{\lambda} \|x_{k+1} - x_k\|^2 + \frac{1}{2\lambda} \|x_{k+1} - x_k \|^2 = \\
		&\leq F(z) - G_\lambda(x_k)^T (z - x_k) - \frac{1}{2\lambda} \|x_{k+1} - x_k\|^2,
	\end{align*}
	which proves the lemma. Alternatively, the last term in the last row of the inequality can be rewritten as $\frac{\lambda}{2}\| G_\lambda(x_k)\|^2$.
\end{proof}

We are now ready to prove the following convergence result,  \cite{Sarkar2015}.

\begin{theorem}\label{thm:3.1}
	Let $F(x) = f(x) + g(x)$ with $f : \mathbb{R}^d \to \mathbb{R}$ continuously differentiable, convex and $L$-smooth, and $g : \mathbb{R}^d \to \mathbb{R}$ continuous, convex and possibly non-smooth. Let $0 < \lambda \leq \frac{1}{L}$ and let $\{x_k\}$ be the sequence generated by Algorithm~\ref{alg:01}. Then for any $n \geq 1$ we have:
	\begin{equation*}
		F(x_n) - F(x^*) \leq \frac{1}{2n\lambda} \|x_0 - x^*\|^2,
	\end{equation*}
	where $x^*$ is a solution to \eqref{ProxUOProblem}.
\end{theorem}

\begin{proof}
	Setting $z = x_k$ in the result of Lemma~\ref{sec03:lem01} we get:
	\begin{equation}\label{sec03:eq06}
		F(x_{k+1}) \leq F(x_k) - \frac{1}{2\lambda} \|x_{k+1} - x_k \|^2,
	\end{equation}
	i.e. the sequence of generated function values for $F$ is monotonically decreasing:
	\[
		F(x_{k+1}) \leq F(x_k) \leq \ldots \leq F(x_0).
	\]
	Similarly, setting $z = x^*$ in the lemma we get:
	\begin{align}
		F(x_{k+1}) - F(x^*) &\leq G_\lambda(x_k)^T (x_k - x^*) - \frac{1}{2\lambda} \|x_{k+1} - x_k\|^2 = \notag \\
		&= \frac{1}{2\lambda} \left[ 2\lambda G_\lambda(x_k)^T (x_k - x^*) - \|x_{k+1} - x_k \|^2 \right] = \notag \\
		&= \frac{1}{2\lambda} \left[ 2\lambda G_\lambda(x_k)^T (x_k - x^*) - \|\lambda G_\lambda(x_k) \|^2 \right] = \notag \\
		&= \frac{1}{2\lambda} \left[ \|x_k - x^*\|^2 - \|x^* - (x_k - \lambda G_\lambda(x_k))\|^2 \right] = \notag \\
		&= \frac{1}{2\lambda} \left[ \|x_k - x^*\|^2 - \|x_{k+1} - x^*\|^2 \right]. \label{sec03:eq07}
	\end{align}
	Summing the inequalities in \eqref{sec03:eq07} for $k=0, 1, \ldots, n-1$ we get:
	\[
		\sum_{k=0}^{n-1} \big( F(x_{k+1}) - F(x^*) \big) \leq \frac{1}{2\lambda} \|x_0 - x^*\|^2,
	\]
	while taking \eqref{sec03:eq06} into consideration, we have:
	\[
	n\big( F(x_n) - F(x^*) \big) = \sum_{k=0}^{n-1} \big( F(x_n) - F(x^*) \big) \leq \sum_{k=0}^{n-1} \big( F(x_{k+1}) - F(x^*) \big) \leq \frac{1}{2\lambda} \|x_0 - x^*\|^2,
	\]
	from where the result follows by dividing through by $n$.
\end{proof}

\begin{remark}
	Note that Theorem~\ref{thm:3.1} holds for the constant step size $\lambda = \frac{1}{L}$.
\end{remark}

Theorem~\ref{thm:3.1} assesses the complexity of the proximal GD method given by Algorithm~\ref{alg:01}, and demonstrates that it is $\mathcal{O}(1/\varepsilon)$, i.e. it is of the same order as the complexity of the GD method, as shown in Theorem~\ref{sec02:thm01}. Note however, that the proximal GD method is more involved than the GD method since it requires the additional effort of calculating the proximal mapping in order to generate the iterates. In general, this is no trivial task, however it is possible to obtain a closed form solution in some special cases.

\subsection{Convergence properties when $f$ is $\mu$-strongly convex}

We will show that under an additional assumption that $f$ is $\mu$-strongly convex the convergence rate of the proximal GD method improves significantly. In the proof of Lemma~\ref{sec03:lem01} we showed that $x_{k+1} = \mathrm{prox}_\lambda \big( x_k - \lambda \nabla f(x_k) \big) = x_k - \lambda G_\lambda (x_k)$, where $G_\lambda (x_k) = \nabla f(x_k) + \gamma_{k+1}$ with $\gamma_{k+1} \in \partial g(x_{k+1})$. First we prove that the progress of the iterates $\{x_k\}$ to the solution $x^*$ can be bounded by the norm of $G_\lambda$, and the strong convexity parameter $\mu$. We have the following lemma, \cite{Denizcan}.

\begin{lemma}\label{sec03:lem02}
	Let $F(x) = f(x) + g(x)$ with $f : \mathbb{R}^d \to \mathbb{R}$ continuously differentiable, $L$-smooth and $\mu$-strongly convex, $g : \mathbb{R}^d \to \mathbb{R}$ continuous, convex and possibly non-smooth. Then the distance of the iterates $\{x_k\}$, generated by Algorithm~\ref{alg:01}, to the optimal solution $x^*$ of \eqref{ProxUOProblem}, can be bounded as:
	\[
		\|x_k - x^*\| \leq \frac{2}{\mu} \| G_\lambda (x_k) \|,
	\]
	for any $k \geq 0$ and $0 < \lambda \leq \frac{1}{L}$, where $G_\lambda (x_k) = \nabla f(x_k) + \gamma_{k+1}$, and $\gamma_{k+1} \in \partial g(x_{k+1})$.
\end{lemma}

\begin{proof}
	The firm non-expansiveness property holds for the proximal mapping (see Beck \cite{Beck}, pg. 158), i.e. for any $x, y \in \mathbb{R}^d$:
	\[
		\| \mathrm{prox}_\lambda (x) - \mathrm{prox}_\lambda (y) \|^2 \leq \big( \mathrm{prox}_\lambda (x) - \mathrm{prox}_\lambda (y) \big)^T(x - y).
	\]
	Putting $x = x_k - \lambda \nabla f(x_k)$ and $y = x^* - \lambda \nabla f(x^*)$, and by (\ref{PGDMethod}) and (\ref{fixedpoint}), we obtain:
	\begin{align*}
		\| x_{k+1} - x^* \|^2 &\leq (x_{k+1} - x^*)^T (x_k - \lambda \nabla f(x_k) - x^* + \lambda \nabla f(x^*)) \\ 
		\Longleftrightarrow \|x_k - \lambda G_\lambda (x_k) - x^*\|^2 &\leq (x_k - \lambda G_\lambda (x_k) - x^*)^T(x_k - \lambda \nabla f(x_k) - x^* + \lambda \nabla f(x^*)) \\
		\Longleftrightarrow 0 &\leq (x_k - \lambda G_\lambda (x_k) - x^*)^T (G_\lambda (x_k) + \nabla f(x^*) - \nabla f(x_k)),
	\end{align*}
	from where with further expansion we get:
	\begin{align}
		(x_k - x^*)^T (\nabla f(x_k) - \nabla f(x^*)) &\leq (x_k - x^*)^T G_\lambda (x_k) + \lambda G_\lambda (x_k)^T (\nabla f(x^*) - \nabla f(x_k)) \leq \notag \\
		&\leq \|G_\lambda (x_k)\| \|x_k - x^*\| + \lambda \|G_\lambda (x_k)\| \|\nabla f(x^*) - \nabla f(x_k)\| = \notag \\
		&= \|G_\lambda (x_k)\| \big( \|x_k - x^*\| + \lambda \|\nabla f(x^*) - \nabla f(x_k)\| \big) \leq \notag \\
		&\leq \|G_\lambda (x_k)\| \big( \|x_k - x^*\| + \lambda L\|x_k - x^*\| \big) \leq \notag \\
		&\leq (1+\lambda L) \|G_\lambda (x_k)\| \|x_k - x^*\| \leq \notag \\
		&\leq 2 \|G_\lambda (x_k)\| \|x_k - x^*\|. \label{sec03:eq09}
	\end{align}
	On the other hand, the strong convexity of $f$ yields:
	\[
		\mu \|x_k - x^*\|^2 \leq (x_k - x^*)^T(\nabla f(x_k) - \nabla f(x^*)),
	\]
	so combining the last inequality with \eqref{sec03:eq09} gives the result.
\end{proof}

Now we prove the main convergence result for a $\mu$-strongly convex $f$, \cite{Denizcan}.

\begin{theorem}\label{thm:ExpConvProxGrad}
	Let $F(x) = f(x) + g(x)$ with $f : \mathbb{R}^d \to \mathbb{R}$ continuously differentiable, $L$-smooth and $\mu$-strongly convex, and $g : \mathbb{R}^d \to \mathbb{R}$ continuous, convex and possibly non-smooth. Let $\{x_k\}$ be the sequence generated by Algorithm~\ref{alg:01}. Then for $0 < \lambda \leq \frac{1}{L}$ the following holds for all $k \geq 1$:
	\begin{equation*}
		F(x_k) - F(x^*) \leq \left(1 + \frac{\lambda \mu}{4}\right)^{-k} \big( F(x_0) - F(x^*) \big),
	\end{equation*}
where $x^*$ is a solution to \eqref{ProxUOProblem}.
\end{theorem}

\begin{proof}
	Since $f$ is $L$-smooth, subtracting two subsequent function values gives:
	\begin{align}
		F(x_{k+1}) - F(x_k) &= f(x_{k+1}) - f(x_k) + g(x_{k+1}) - g(x_k) \leq \notag \\
		&\leq \nabla f(x_k)^T (x_{k+1} - x_k) + \frac{L}{2}\|x_{k+1} - x_k \|^2 + g(x_{k+1}) - g(x_k) = \notag \\
		&= -\lambda \nabla f(x_k)^T G_\lambda (x_k) + g(x_{k+1}) - g(x_k) + \lambda^2 \frac{L}{2}\|G_\lambda (x_k)\|^2. \label{sec03:eq10}
	\end{align}
	Also, since $g$ is convex, using the definition of $G_\lambda$ we have:
	\begin{align*}
		-\lambda \nabla f(x_k)^T G_\lambda (x_k) &+ g(x_{k+1}) - g(x_k) = -\lambda (G_\lambda (x_k) - \gamma_{k+1})^T G_\lambda (x_k) + g(x_{k+1}) - g(x_k) = \\
		&= -\lambda \|G_\lambda (x_k)\|^2 + \lambda \gamma_{k+1}^T G_\lambda (x_k) + g(x_{k+1}) - g(x_k) = \\
		&= -\lambda \|G_\lambda (x_k)\|^2 - \gamma_{k+1}^T (x_{k+1} - x_k) + g(x_{k+1}) - g(x_k) \leq \\
		&\leq -\lambda \|G_\lambda (x_k)\|^2.
	\end{align*}
	Combining the last inequality and \eqref{sec03:eq10} we get the following bound using Lemma~\ref{sec03:lem02} and $\lambda \leq \frac{1}{L}$:
	\begin{align}
		F(x_{k+1}) - F(x_k) &\leq -\lambda \|G_\lambda (x_k)\|^2 + \lambda^2 \frac{L}{2}\|G_\lambda (x_k)\|^2 \leq -\frac{\lambda}{2} \|G_\lambda (x_k)\|^2 \leq \notag\\
		& \leq -\frac{\lambda \mu}{4} \|G_\lambda (x_k)\| \|x_{k} - x^*\| \leq -\frac{\lambda \mu}{4} G_\lambda (x_k)^T (x_k - x^*). \label{sec03:eq11}
	\end{align}
	Since $f$ and $g$ are both convex, we can write the product on the right-hand side as:
	\begin{align}
		-G_\lambda (x_k)^T (x_k - x^*) &= -\big( \nabla f(x_k) + \gamma_{k+1} \big)^T (x_k - x^*) = \notag \\
		&= -\nabla f(x_k)^T (x_k - x^*) - \gamma_{k+1}^T (x_k - x_{k+1} + x_{k+1} - x^*) = \notag \\
		&= -\nabla f(x_k)^T (x_k - x^*) + \gamma_{k+1}^T (x_{k+1} - x_k) - \gamma_{k+1}^T (x_{k+1} - x^*) = \notag \\
		&= -\nabla f(x_k)^T (x_k - x^*) - \lambda \gamma_{k+1}^T G_\lambda (x_k) - \gamma_{k+1}^T (x_{k+1} - x^*) \leq \notag \\
		&\leq f(x^*) - f(x_k) + g(x^*) - g(x_{k+1}) - \lambda \gamma_{k+1}^T G_\lambda (x_k). \label{sec03:eq12}
	\end{align}
	We bound the last term on the right-hand side of \eqref{sec03:eq12} using the $L$-smoothness of $f$ as following:
	\begin{align}
		-\lambda\gamma_{k+1}^T G_\lambda (x_k) &= - \lambda \big( G_\lambda (x_k) -\nabla f(x_k) \big)^T G_\lambda (x_k) = \notag \\
		&= -\lambda \|G_\lambda(x_k)\|^2 - \nabla f(x_k)^T (x_{k+1} - x_k) \leq \notag \\
		&\leq -\lambda \|G_\lambda(x_k)\|^2 - f(x_{k+1}) + f(x_k) +  \frac{L}{2}\|x_{k+1} - x_k \|^2 \leq \notag \\
		&\leq f(x_k) - f(x_{k+1}) - \frac{\lambda}{2} \|G_\lambda(x_k)\|^2. \notag
	\end{align}
	Using this bound in \eqref{sec03:eq12} gives:
	\begin{align}
		-G_\lambda (x_k)^T (x_k - x^*) &\leq f(x^*) + g(x^*) - g(x_{k+1}) - f(x_{k+1}) - \frac{\lambda}{2} \|G_\lambda(x_k)\|^2 = \notag\\
		&= F(x^*) - F(x_{k+1}) - \frac{\lambda}{2} \|G_\lambda(x_k)\|^2 \leq \notag \\
		&\leq F(x^*) - F(x_{k+1}), \label{sec03:eq13}
	\end{align}
	while substituting \eqref{sec03:eq13} in \eqref{sec03:eq11} yields:
	\begin{align*}
		F(x_{k+1}) - F(x_{k}) &\leq \frac{\lambda \mu}{4} \big( F(x^*) - F(x_{k+1}) \big) \\
		\Longleftrightarrow \frac{\lambda \mu}{4} \big( F(x_{k+1}) - F(x^*) \big) &\leq F(x_{k}) - F(x_{k+1}).
	\end{align*}
	Adding and subtracting $F(x^*)$ on the right-hand side gives the bound:
	\begin{equation}\label{sec03:eqLast}
		F(x_{k+1}) - F(x^*) \leq \left( 1 + \frac{\lambda \mu}{4} \right)^{-1} \big( F(x_k) - F(x^*) \big),
	\end{equation}
	from where the result of the theorem follows recursively.
\end{proof}

Theorem~\ref{thm:ExpConvProxGrad} demonstrates that the complexity of the proximal GD method under the assumption of $\mu$-strong convexity for the function $f$ is of order $\mathcal{O} (\log{(1/\varepsilon)})$ which is a noticeable improvement of the result from Theorem~\ref{thm:3.1}.

\subsection{Proximal mapping of least absolute shrinkage and selection operator (LASSO) regularization}

We now demonstrate how the proximal GD method can be constructed in the case of LASSO or $\ell^1$ regularization. In this context the function $g$ in the definition of problem \eqref{ProxUOProblem} is:
\begin{equation}\label{sec03:l1reg}
	g(x) = \alpha \|x \|_1 = \alpha \sum_{i=1}^d |x_i|,
\end{equation}
where $\alpha > 0$ is a \textit{regularization parameter}. Regularization procedures are frequently applied in practice to limit the magnitude of the components of the solution. In particular, the LASSO is one of the three most frequently used regularization techniques in regression modeling (the other two being the Ridge $\ell^2$ regularization, and the elastic-net regularization). The LASSO regularization has the added property of eliminating non-influential or non-important regression variables from the solution by \textit{annulling} the coefficients of the appropriate components in the solution. As can be seen from the definition of $g$, using $1$-norm results in $g$ not being differentiable at some points; this makes proximal GD method an ideal candidate for constructing iterative methods for solving this type of regularized problems.

For convenience, let us define a function $\varphi(y, z)$ from \eqref{sec03:ProximalMapping} with $g(y) = \alpha \|y\|_1$ as:
\[
	\varphi(y, z) =  \frac{1}{2\lambda} \left\| y - z \right\|^2 + \alpha \|y\|_1
	= \frac{1}{2\lambda} \sum_{i=1}^d (y_i - z_i)^2 + \alpha \sum_{i=1}^d |y_i|
\]
The proximal mapping then gets the following form:
\[
	\mathrm{prox}_\lambda (z) = \mathop{\mathrm{argmin}}_{y\in\mathbb{R}^d} \varphi(y, z).
\]
To solve this problem, we differentiate $\varphi$ with respect to $y_i$, and equating the partial derivatives to zero gives:
\begin{align*}
	\frac{\partial \varphi}{\partial y_i}(y, z) &= \frac{1}{\lambda} (y_i - z_i) + \alpha\, \mathrm{sgn}(y_i) = 0 \\
	\Longrightarrow z_i &= y_i + \alpha\lambda\, \mathrm{sgn}(y_i)
	= \begin{cases}
		y_i + \alpha\lambda, & y_i > 0 \\
		y_i - \alpha\lambda, & y_i < 0
	\end{cases},
\end{align*}
where we conveniently used that $\frac{d}{dw}(|w|) = \mathrm{sgn}(w)$ for $w\neq 0$. From the solutions for $z_i$ we can obtain expressions for $y_i$ and formulate an explicit form for the proximal mapping:
\begin{equation}\label{sec03:eq08}
	\mathrm{prox}_\lambda (z) = 
	\begin{cases}
		z_i - \alpha\lambda, & z_i > \alpha\lambda \\
		0, & |z_i| \leq \alpha\lambda \\
		z_i + \alpha\lambda, & z_i < -\alpha\lambda
	\end{cases}
	= \mathrm{sgn}(z)\cdot \max \big\{ |z| - \alpha\lambda,\, 0 \big\},
\end{equation}
where the sign, the absolute value, and the maximum functions are all \textit{applied component-wise} to $z$. The function defined by \eqref{sec03:eq08} is known as \textit{soft-} or \textit{shrinkage-thresholding operator}.

Let us note that the soft-thresholding as solution for the proximal mapping in the case of a LASSO regularization allows for a cheap implementation of the proximal GD method, which in turn allows us to circumvent the possible issues of non-differentiability of the $1$-norm.

\section{Proximal Gradient Descent Method with Variable Step Sizes}\label{sec04}

In the previous sections we analyzed the convergence properties of the GD and the proximal GD methods. The underlying result we used throughout the analysis, under appropriate assumptions about the objective function, was that the iterative method leads to maximum decrease if in \eqref{GDMethod} we use a constant step size $\lambda_k \equiv \lambda = \frac{1}{L}$ for $k = 0, 1, \ldots$.

One must, however, note that the $L$-smoothness as a geometric property is not necessarily global. We also note that a convex function is $L$-smooth if and only if its gradient is Lipschitz continuous with constant $L$; see e.g. \cite[Lemma 2.5]{Gartner2023}. For example, the function $f_1(x) = \sqrt{x^2 + 1}$ is globally $L$-smooth since its first derivative $f_1'(x) = x/\sqrt{x^2 + 1}$ is Lipschitz continuous on $\mathbb{R}$ with constant $L_1 = 1$. On the other hand, the function $f_2(x) = x^4$ is not globally $L$-smooth since its first derivative $f_2'(x) = 4x^3$ is not Lipschitz continuous on $\mathbb{R}$. However, on a bounded interval $[a, b] \subset \mathbb{R}$ the function $f_2$ is \textit{locally $L$-smooth} since $f_2'$ is bounded by the constant $L_2 = 4 \max{\{|a|^3, |b|^3\}}$. This discussion implies that $L$-smoothness is a \textit{local property} of functions. In light of the step size selection in the descent methods this implies that the local value of the smoothness constant will likely change throughout the iterations. Thus, it is only natural to explore ways of constructing approximations to the local smoothness constant in neighborhoods of the iterates. These constants can then be used to construct step sizes which, potentially, improve the performance of the descent methods.

In this section we propose a way in which we can construct approximations of the local smoothness constants for the objective function with the intent of using them to calculate step sizes. We still assume that $f$ is \textit{globally} $L$-smooth. In the $k$-th iteration of the method, once $x_{k+1}$ is calculated, we assume $f$ to be \textit{locally} smooth with constant $L_k$ not necessarily equal to $L$, i.e. that $\|\nabla f(x_{k+1}) - \nabla f(x_k)\| \leq L_k \|x_{k+1} - x_k \|$. Since locally the best choice of the step size is still $\lambda_k = 1/L_k$, we \blue{could} use:
\begin{equation}\label{sec04:eq01}
	\lambda_k = \frac{1}{L_k} \leq \frac{\|x_{k+1} - x_k \|}{\|\nabla f(x_{k+1}) - \nabla f(x_k)\|}.
\end{equation}
The approach we propose is inspired by similar approaches outlined in \cite{LiuWangLiu} and \cite{Malitsky}.
\blue{
	In particular, the authors of \cite{Malitsky} use an approximation of a locally optimal step size with suitable correction in order to establish the convergence of the method:
	\begin{equation}
	\begin{aligned}
		\lambda_{k+1} &= \min \left\{ \sqrt{1 + \theta_{k}} \lambda_{k}, \frac{\| x_{k+1} - x_{k}\|}{2 \| \nabla f(x_{k+1}) - \nabla f(x_{k}) \|} \right\} \\
		\theta_{k+1} &= \frac{\lambda_{k+1}}{\lambda_{k}},
	\end{aligned}
	\end{equation}
	with $\lambda_0 > 0$ and $\theta_0 = +\infty$. These step sizes are then used within the usual GD scheme, and within a Nesterov-type accelerated GD scheme.
	
	On the other hand, the authors of \cite{LiuWangLiu} take an approach which approximates the local Lipschitz constant of the objective function with:
	\begin{equation}
		\lambda_{k+1} = \frac{\mu_1 \| x_{k} - y_{k} \|^2}{2 |f(x_k) - f(y_k) - \nabla f(x_k)^T (x_k - y_k)|},
	\end{equation}
	if $2 |f(x_k) - f(y_k) - \nabla f(x_k)^T (x_k - y_k)| > \frac{\mu_0}{\lambda_k} \| x_{k} - y_{k} \|^2$ holds, otherwise:
	\begin{equation}
		\lambda_{k+1} = \lambda_{k} + \min{\left\{ \lambda_{k}, 1 \right\} \cdot \eta_k},
	\end{equation}
	where $0 < \mu_1 < \mu_0 < 1$, $\sum{\eta_k} < \infty$, $\eta_k > 0$, and $y_k$ is a test point within a Nesterov-type accelerated proximal GD scheme. This proves helpful within the non-convexity framework they discuss.
}

\blue{The proposed variable step size proximal GD method, based on \eqref{sec04:eq01}, was developed based on the following rationale: starting with an initial step size $\lambda_0 > 0$, we iteratively generate step sizes ensuring that their magnitude remains smaller than the norm ratio on the right-hand side of \eqref{sec04:eq01}; however, we must also take care not to generate too small step sizes which will slow down the process. \blue{To achieve this} we fix two constants $0 < \mu_1 < \mu_0 < 1$ and a sequence $\{\eta_k\}$ such that $\eta_k>0$ and $\sum\eta_k < \infty$. To update $\lambda_k$ to $\lambda_{k+1}$ in the $k$-th iteration we test if:
	\begin{equation}
		\lambda_k > \frac{\mu_0 \|x_{k+1} - x_k \|}{\|\nabla f(x_{k+1}) - \nabla f(x_k)\|}.
	\end{equation}
If the test is true, then $\lambda_k$ is either greater than the norm ratio, or it is smaller than, but too close to it, i.e. its value is greater than $\mu_0 \cdot 100 \%$ of the value of the norm ratio. In such case the step size should be decreased to an acceptable size, as per \eqref{sec04:eq01}, so the next step size is selected as:
	\begin{equation}
		\lambda_{k+1} = \frac{\mu_1 \|x_{k+1} - x_k \|}{\|\nabla f(x_{k+1}) - \nabla f(x_k)\|},
	\end{equation}
i.e. it is set to $\mu_1\cdot 100\%$ of the norm ratio, which ensures $\lambda_{k+1} < \lambda_{k}$. In case the test is false, the current step size $\lambda_k$ might be too small, so we perform a controlled increase of its magnitude using the sequence $\{\eta_k\}$ by setting $\lambda_{k+1} = \lambda_k + \min{\{\lambda_k, 1\}} \cdot \eta_k$. This variable step size scheme can then be incorporated in any gradient descent scheme. The concrete choices of the constants $\mu_0$ and $\mu_1$ allow for a control over how large departures from the norm ratio from \eqref{sec04:eq01} we are willing to tolerate in the step size generation process. The full description is given in Algorithm~\ref{alg:02}. The test in line 5 is implemented through an equivalent statement for numerical stability.}

\begin{algorithm}[t]
	\caption{Proximal Gradient Descent Method with Variable Step Size}
	\begin{algorithmic}[1]
		\State \textbf{input:} $x_0 \in \mathbb{R}^d$, $\lambda_0 >0$, $0 < \mu_1 < \mu_0 < 1$, $N \in \mathbb{N}$, $\{\eta_k\}$ s.t. $\sum \eta_k < \infty$
		\State \textbf{set:} $k=0$
		\While{$k < N$}
		\State $x_{k+1} = \mathrm{prox}_{\lambda_k} \big(x_k - \lambda_k \nabla f(x_k) \big)$
		\If{$\|\nabla f(x_{k+1}) - \nabla f(x_k) \| > \frac{\mu_0}{\lambda_k} \| x_{k+1} - x_k\|$}
		\State $\lambda_{k+1} = \dfrac{\mu_1 \| x_{k+1} - x_k\|}{\|\nabla f(x_{k+1}) - \nabla f(x_k) \|}$
		\Else 
		\State $\lambda_{k+1} = \lambda_{k} + \min{\{ \lambda_k, 1\}}\cdot \eta_k$
		\EndIf
		\State $k \leftarrow k + 1$
		\EndWhile
	\end{algorithmic}
	\label{alg:02}
\end{algorithm}

The proposed algorithm yields a sequence of values which are bounded away from $1/L_k$ ensuring an \textit{almost optimal} selection of the step sizes as $\lambda_k = 1/L_k$ throughout the iteration process. Moreover, it can be shown that the sequence $\{\lambda_k\}$ is bounded from above and below away from zero, similarly as in \cite{LiuWangLiu}. As illustrated in section~\ref{sec05}, Algorithms \ref{alg:01} and \ref{alg:02} have similar complexities of order $\mathcal{O}(\log{(1/\varepsilon)})$ for the case of $L$-smooth, $\mu$-strongly convex functions.

\section{Numerical Results}\label{sec05}

In this section we summarize the numerical results which illustrate the performance of the methods we analyzed. We have tested the performance of the proximal GD method, both with constant and with variable step sizes, in the context of ordinary least squares (OLS) for linear regression with $\ell^1$ regularization, as described and discussed in section \ref{sec03}. In the case of constant step sizes in Algorithm~\ref{alg:01} we used $\lambda = 1/L$. In the case of variable step sizes in Algorithm~\ref{alg:02} \blue{we have empirically established that appropriate values for the parameters are: $\lambda_0 = 0.1$, $\mu_0 = 0.99$, and $\mu_1 = 0.95$. Additionally, we have compared the performance of these proximal GD methods with the performance of the Adam algorithm \cite{KingmaBa} adapted for $\ell^1$ regularization using \eqref{sec03:l1reg}. All of Adam's parameters were set to their recommended values, as outlined in the original paper \cite{KingmaBa}.}

\blue{We evaluated the performance of the proposed method using two datasets: (a) a synthetically generated data set and (b) a real-data set, specifically the King County real estate dataset available in the OpenML database \cite{Schmitt2019}. The \textit{synthetically generated data set}} is characterized by: number of variables/dimensionality $d$, sample size $m$, and number of non-zero coordinates in the optimal solution (i.e. a \textit{sparsity indicator}) $s$. These quantities are roughly related as $s \ll d \ll m$. First, we generated an optimal solution $x^* \in \mathbb{R}^m$ such that only its first $s$ coordinates were non-zero, with the $i$-th coordinate randomly generated in the interval $(0, 1)$ for $i=1, 2, \ldots, s$. Next, we generated the ``main'' data set as a random matrix $A \in \mathbb{R}^{m\times d}$ with correlated standard normal entries generated using a correlation matrix $C\in\mathbb{R}^{d\times d}$ where $c_{ij} = 0.5^{|i-j|}$ for $i, j=1, 2, \ldots, d$. Finally, the target vector $b\in\mathbb{R}^m$ was calculated as $b = Ax^* + \xi$ where the random vector $\xi = (\xi_1, \ldots, \xi_m)^T\in\mathbb{R}^m$ with $\xi_1, \ldots, \xi_m$ i.i.d. standard normally distributed random variables. The objective function used in all tests was the usual OLS function with $\ell^1$ regularization \blue{for linear regression}:
\begin{equation}\label{sec05:eq01}
	F(x) = \underbrace{\frac{1}{2m} \|Ax - b\|^2}_{f(x)} + \underbrace{\alpha \|x\|_1}_{g(x)}.
\end{equation}
Under these conditions we generated three data sets such that:
\[
	(d, m, s) \in \big\{ (300, 30\,000, 30), (500, 50\,000, 50), (800, 80\,000, 80) \big\}
\]
\blue{
	The \textit{King County real estate data set} is a standard regression model analysis data set accessible through OpenML's online database \cite{Schmitt2019}. The goal was to model the selling price of property/houses in King County based on the houses' characteristics and geographical position. In total, there were 18 data features which we used as independent variables in the model. No data were missing from the data set. For the purpose of model evaluation, the data was randomly split into a train and test data sets, while preprocessing included only standardization of all model variables.
}

The value of the regularization parameter $\alpha$ was set to $\alpha=0.01$, while the maximum number of iterations was set to $N = 1000$ in all tests. The smoothness constant $L$ was calculated as the largest eigenvalue of the matrix $\frac{1}{2m}A^TA$.

Apart from the obvious stopping criterion of reaching the maximum number of iterations, the methods also stopped if the norm of the gradient of the smooth part $f$ fell below a threshold of $0.001$, i.e. if $\|\nabla f (x_k) \| < 0.001$\blue{; additionally, the proximal GD methods (Algorithm~\ref{alg:01} and Algorithm~\ref{alg:02}) stopped} if the sequence of function values started displaying non-monotone behavior, i.e. $F(x_{k+1}) > F(x_k)$ at some iteration $k$. This last condition proves useful in early stopping since it only seems to ``activate'' in presence of noise close to the optimal solution; \blue{since Adam is non-monotone by design, this stopping criterion was not used with Adam}. For Algorithm~\ref{alg:01} and Algorithm~\ref{alg:02} we measured the average CPU time required, out of $7$ runs under identical initial conditions, for the optimization to stop. The tests and simulations were coded in \textit{Python} and all testing was performed via \textit{Google Colab}. Only default available computational resources allocated to users were used (relevant for the reported CPU times quantities). To make the execution times comparable for the proximal GD methods (Algorithm~\ref{alg:01} and Algorithm~\ref{alg:02}), it was decided that the variable step size proximal GD method would also include a subroutine for calculation of the Lipschitz constant $L$ (the constant step size proximal GD method anyhow must calculate it), even though it neither need it nor use it.

The detailed results about the number of iterations needed for the solver to stop, time of the run (in seconds, average run time in $7$ repetitions), and its speed (in iterations per second) for the synthetic data sets are given in Table~\ref{tab:01}. As can be seen, the proximal GD method with variable step size requires much fewer iterations to solve the problem, and as a result finds a solution much more quickly than its constant step size counterpart (Algorithm~\ref{alg:01}). However, as can be seen in the table, the \blue{constant step size proximal GD method} has better per iteration speed which is likely due to the fact that the variable step size method must recalculate step sizes at every iteration.

\begin{table}
	\centering
	\begin{tabular}{|c||c|c|c||c|c|c|}
		\cline{2-7}\multicolumn{1}{c||}{} & \multicolumn{3}{c||}{Constant step} & \multicolumn{3}{c|}{Variable step} \bigstrut\\
		\multicolumn{1}{c||}{} & \multicolumn{3}{c||}{(Algorthm~\ref{alg:01})} & \multicolumn{3}{c|}{(Algorthm~\ref{alg:02})} \bigstrut\\
		\hline
		Dimension & Iter  & Time  & Speed & Iter  & Time  & Speed \bigstrut\\
		\hline
		\hline
		$300$   & $152$   & $3.92$  & $38.78$ & $68$    & $2.07$  & $32.85$ \bigstrut\\
		\hline
		$500$   & $181$   & $12.00$  & $15.08$ & $77$    & $5.93$  & $12.98$ \bigstrut\\
		\hline
		$800$   & $229$   & $37.80$  & $6.06$  & $69$    & $14.30$  & $4.83$ \bigstrut\\
		\hline
	\end{tabular}%
	\caption{Performance comparison for the proximal GD methods with constant and with variable step size on the synthetic data sets. \blue{Time is given in seconds, while speed in iterations per second.}}
	\label{tab:01}%
\end{table}%

\blue{As an additional evaluation of the performance, in the case of the synthetic data, we analyze the values of the objective function $F(x_k)$, and the distances to the optimal solution $\|x_k - x^*\|$, visualized on Figure~\ref{fig:comp_f} and Figure~\ref{fig:comp_x} respectively. The comparison on these figures shows the performance of the proximal GD methods (Algorithm~\ref{alg:01} and Algorithm~\ref{alg:02}), and Adam. The better performance of the variable step size proximal GD method is evident on both figures. It is obvious that the variable step size proximal GD method consistently outperforms the constant step size proximal GD method. The poor performance of Adam (along with noticeable instability in the latter iterations) compared to the variable step size proximal GD method is likely due to the fact that, unlike the proximal GD methods, it is not designed to \textit{naturally} work with $\ell^1$ regularization. On the other hand, the good performance of the variable step size proximal GD method} can be attributed to adaptations of the step sizes to the local geometry of the objective function and are sometimes less than, sometimes greater than the \textit{optimal constant step} $1/L$, as shown in Figure~\ref{fig:step_size}.

\begin{figure}
	\centering
	\begin{subfigure}[b]{0.3\textwidth}
		\includegraphics[width=\textwidth]{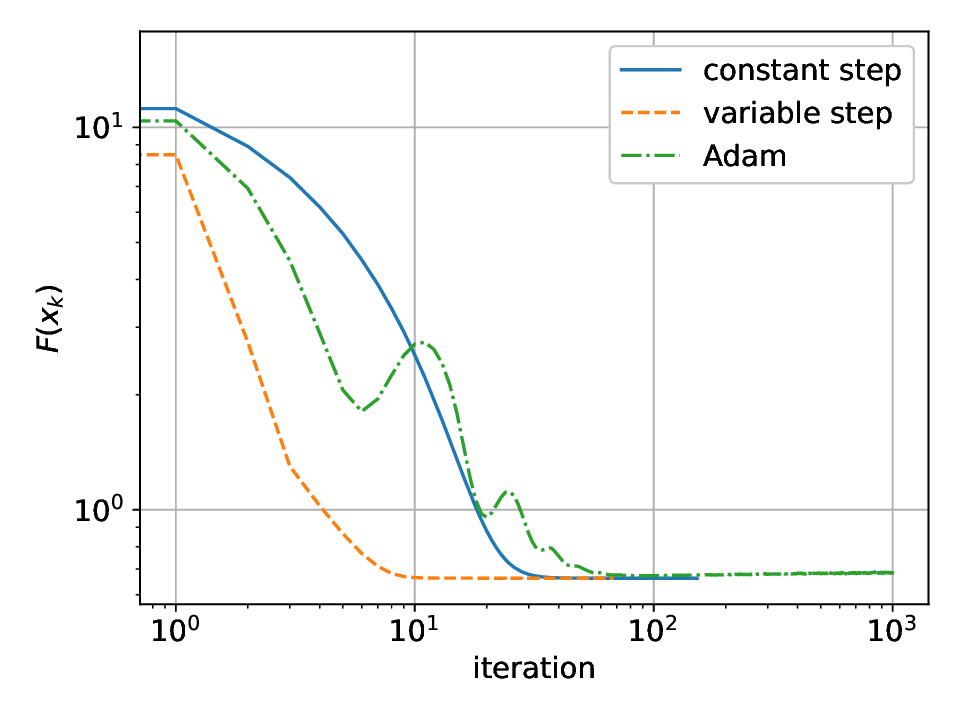}
		\caption{}
		\label{fig:comp_f_a}
	\end{subfigure}
	~ 
	\begin{subfigure}[b]{0.3\textwidth}
		\includegraphics[width=\textwidth]{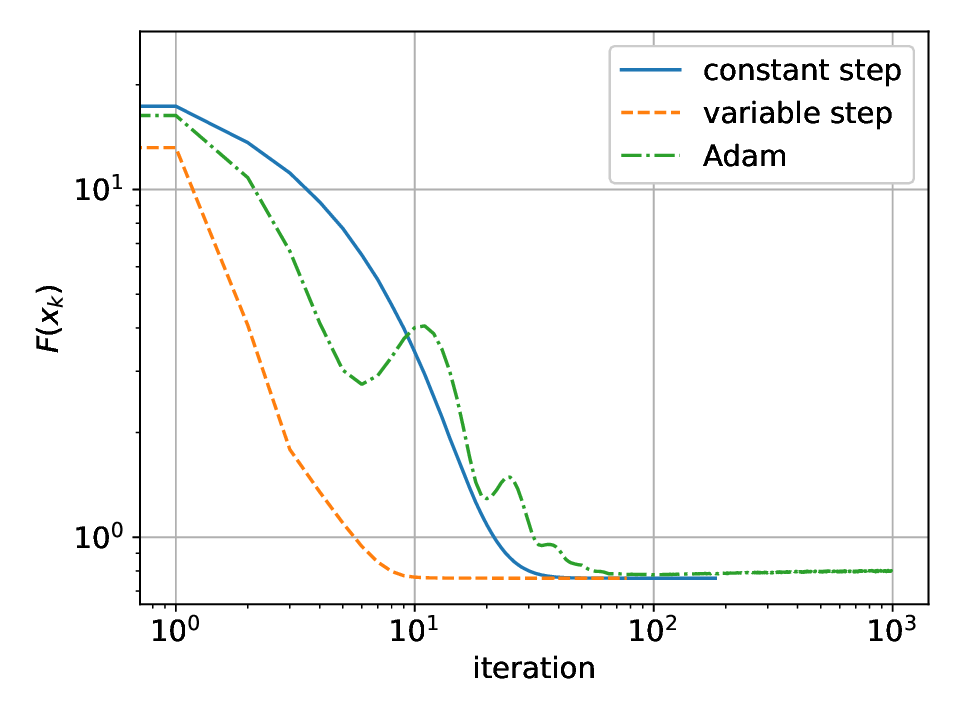}
		\caption{}
		\label{fig:comp_f_b}
	\end{subfigure}
	~ 
	\begin{subfigure}[b]{0.3\textwidth}
		\includegraphics[width=\textwidth]{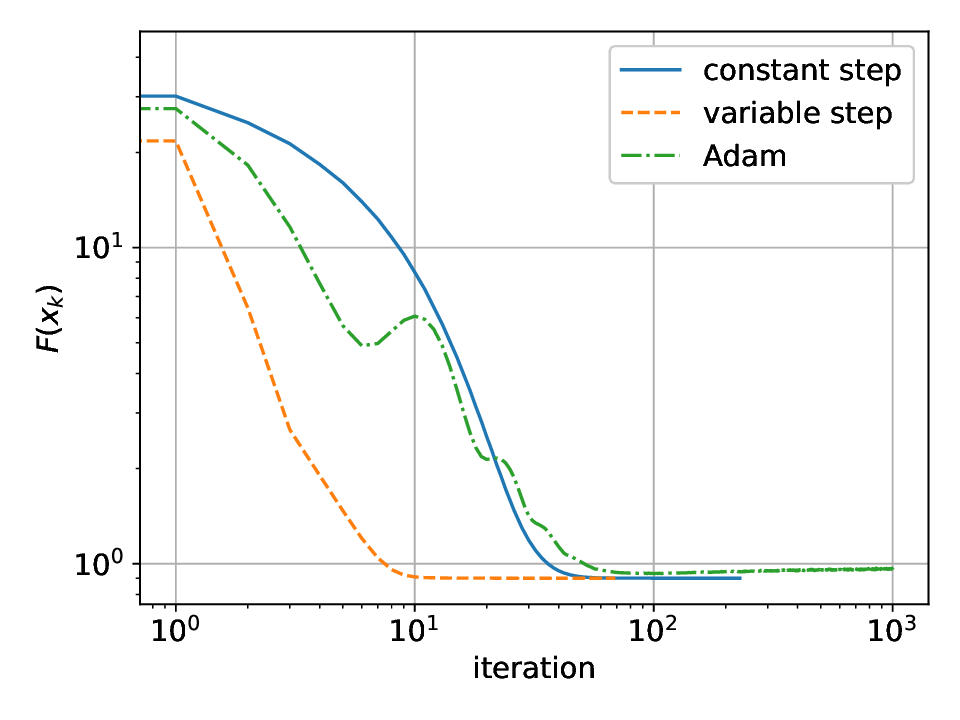}
		\caption{}
		\label{fig:comp_f_c}
	\end{subfigure}
	\caption{Comparison of the progress of the sequence $\{F(x_k)\}$ of function values generated by Algorithm~\ref{alg:01} (constant step), Algorithm~\ref{alg:02} (variable step), and Adam on the synthetically generated data sets for: (a) $d=300$; (b) $d=500$; (c) $d=800$.}
	\label{fig:comp_f}
\end{figure}

\begin{figure}
	\centering
	\begin{subfigure}[b]{0.3\textwidth}
		\includegraphics[width=\textwidth]{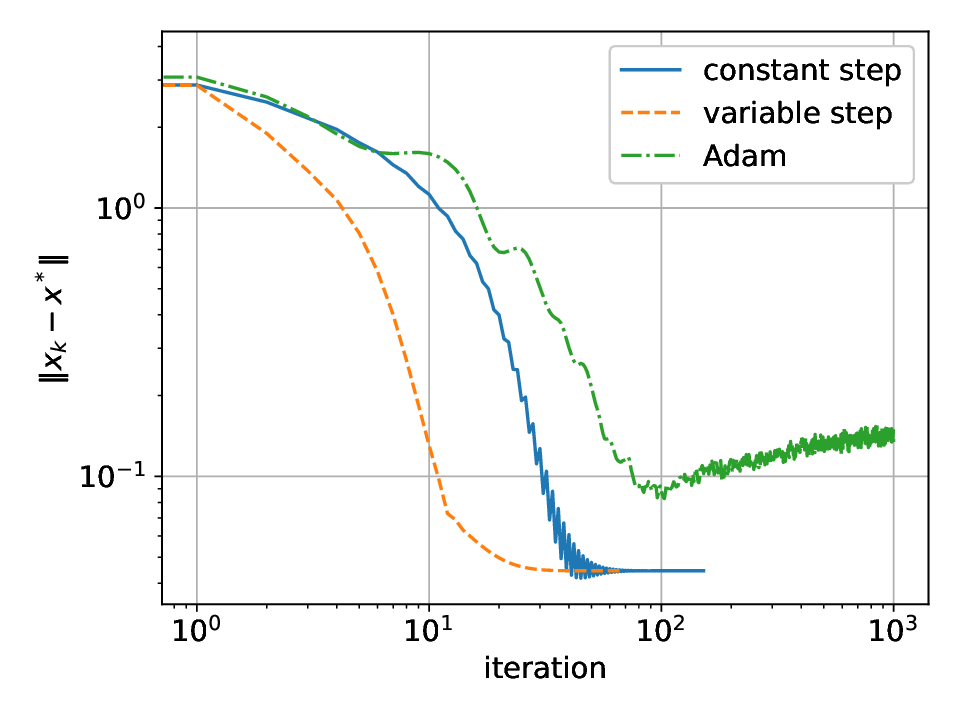}
		\caption{}
		\label{fig:comp_x_a}
	\end{subfigure}
	~ 
	\begin{subfigure}[b]{0.3\textwidth}
		\includegraphics[width=\textwidth]{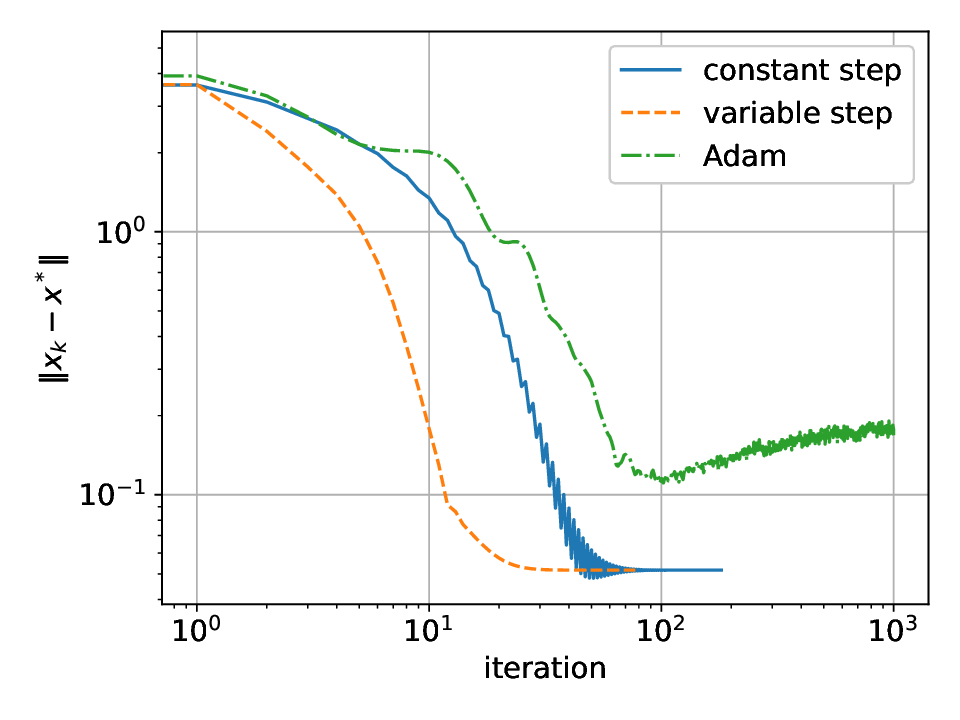}
		\caption{}
		\label{fig:comp_x_b}
	\end{subfigure}
	~ 
	\begin{subfigure}[b]{0.3\textwidth}
		\includegraphics[width=\textwidth]{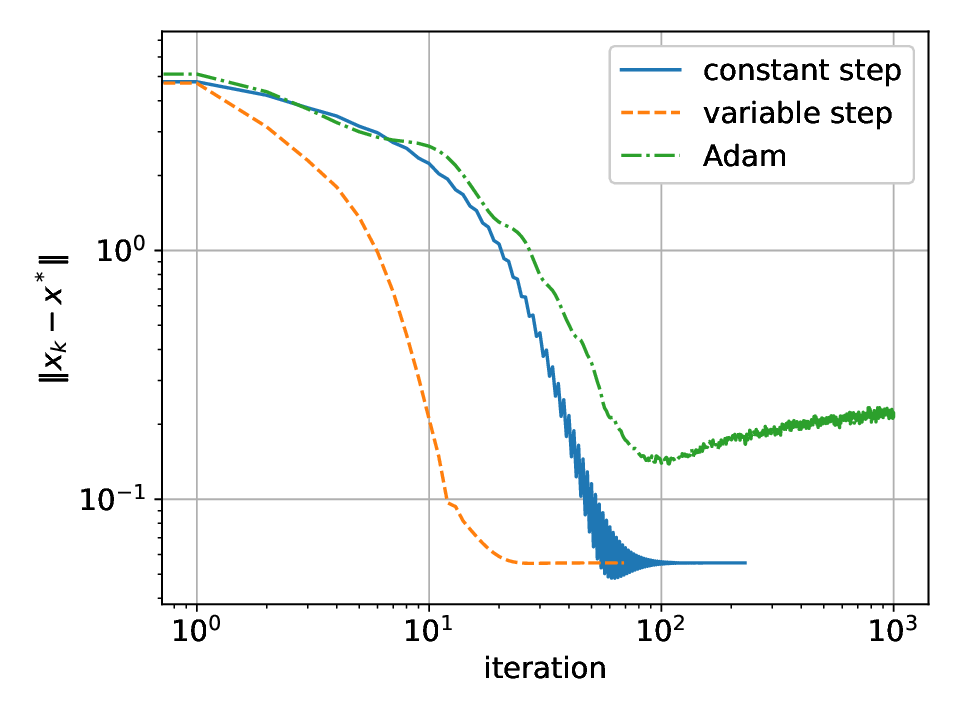}
		\caption{}
		\label{fig:comp_x_c}
	\end{subfigure}
	\caption{Comparison of the progress of the sequence $\{x_k\}$ of iterates generated by Algorithm~\ref{alg:01} (constant step), Algorithm~\ref{alg:02} (variable step), and Adam towards the optimal solution $x^*$ on the synthetically generated data sets for: (a) $d=300$; (b) $d=500$; (c) $d=800$.}
	\label{fig:comp_x}
\end{figure}

\begin{figure}
	\centering
	\includegraphics[width=0.5\textwidth]{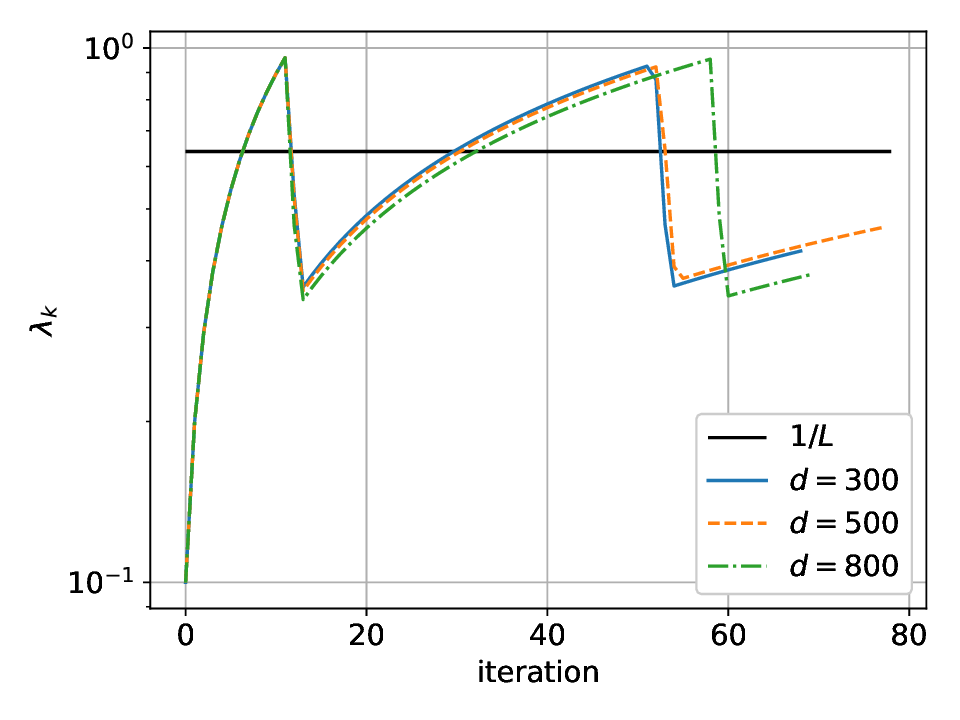}
	\caption{Changes of the step sizes $\lambda_k$ generated by the scheme given in Algorithm~\ref{alg:02} \blue{on the synthetically generated data sets} for different values of $d$. The value of the constant step $\lambda = 1/L$ is given for comparison. For all three dimensions, $1/L \approx 0.64$.}
	\label{fig:step_size}
\end{figure}

\blue{
The models' performance on the King County property/house prices real-data set is slightly different. In this case we can only use the changes of the objective function values $F(x_k)$ as a performance measure, since the optimal value $x^*$ is unknown. These are shown on Figure~\ref{fig:real_data_func_val}. Once again the variable step size proximal GD method outperforms the other two methods, but in this case Adam performs better than the constant step size proximal GD method. It is interesting to note that Adam performs better than the proximal GD with variable step size in the first few iterations.
}

\begin{figure}
	\centering
	\includegraphics[width=0.5\textwidth]{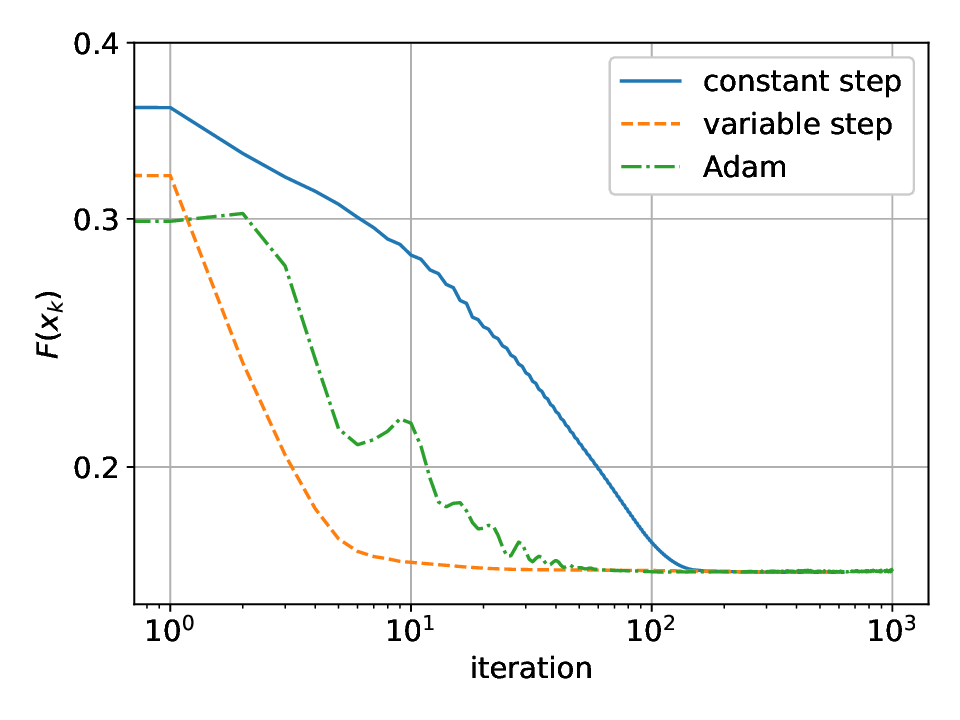}
	\caption{Comparison of the progress of the sequence $\{F(x_k)\}$ of function values generated by Algorithm~\ref{alg:01} (constant step), Algorithm~\ref{alg:02} (variable step), and Adam on the King County house prices data set.}
	\label{fig:real_data_func_val}
\end{figure}

\section{Conclusion}\label{sec06}

In this paper we outlined the basic ideas behind the GD and the proximal GD methods, and we have analyzed their performance.

The initial analysis centered around a constant choice for the step size in the iterative process, setting it to the reciprocal of the smoothness constant $L$. As discussed further in the paper, the magnitude of the smoothness constant is a local rather than a global property, so we looked at one possible way of approximating the local smoothness constants with the goal of using them to calculate locally appropriate step sizes which may vary throughout the iterations. Therefore, a new proximal GD method with variable step sizes was proposed.

Both methods were tested on set of three, large-size, random synthetic data sets, \blue{plus one real-data set}, all in context of an $\ell^1$ regularized OLS problem \blue{for linear regression. Additionally, we compared the performance of the constant and variable step size proximal GD methods with Adam, one of the most commonly used first-order optimizers with momentum-based adaptive step sizes in machine learning}. Numerical results showed that the variable step size proximal GD method \blue{outperforms the constant step size proximal GD by a} margin \blue{in terms of number of iterations and overall time}, thus practically demonstrating the advantage of using local smoothness constants as opposed to a global(ized) one. However, the constant step size proximal GD method has clearly better performance in terms of speed in iterations per second. Even though \blue{Adam has a momentum-based adaptive scheme, and variable step sizes, the proposed variable step size proximal GD method outperforms Adam on both the synthetic and real-data sets}.

As evidenced in numerical testing, the proposed variable step size proximal GD method retains the same order of complexity as the constant step size proximal GD method. Providing theoretical results for these insights would definitely be worth exploring in depth in the future.


\begin{thebibliography}{99}
	\bibitem{Beck} A. Beck, \emph{First-Order Optimization Methods}, MOS-SIAM Series on Optimization, Society for Industrial and Applied Mathematics, Philadelphia, 2017. \url{https://doi.org/10.1137/1.9781611974997}
	
	\bibitem{Cauchy} A. L. Cauchy, Méthode générale pour la résolution des systemes d'équations simultanées, \emph{Comp. Rend. Sci.}, \textbf{25} (1847), pp. 536–538.
	
	\bibitem{Denizcan} V. N. Denizcan, M. G\"{u}rb\"{u}zbalaban and A. Ozdaglar, A Simple Proof for the Iteration Complexity of the Proximal Gradient Algorithm, \emph{9th NIPS Workshop on Optimization for Machine Learning}, Barcelona, Spain, 2016. \url{https://opt-ml.org/oldopt/papers/OPT2016_paper_36.pdf}
	
	\bibitem{Dimovski2017} M. Dimovski and I. Stojkovska, Regularized least-square optimization method for variable selection in regression models, \emph{Matematichki Bilten}, \textbf{41} (1) (2017), pp. 80–100. \url{https://doi.org/10.37560/matbil11700080d}
	
	\bibitem{app5} B. R. Gaines, J. Kim and H, Zhou, Algorithms for Fitting the Constrained Lasso, \emph{Journal of Computational and Graphical Statistics}, \textbf{27} (4) (2018), pp. 861–871. \url{https://doi.org/10.1080/10618600.2018.1473777}
	
	\bibitem{Gartner2023} B. G\"{a}rtner and M. Jaggi, \emph{Optimization for Machine Leaning, Lecture Notes CS-439}, February 2023. \url{https://github.com/epfml/OptML_course/blob/master/lecture_notes/lecture-notes.pdf}
	
	
	\bibitem{Hoerl1970} A. E. Hoelr and R. W. Kennard, Ridge regression: Biased estimation for nonorthogonal problems, \emph{Technometrics}, \textbf{12} (1) (1970), pp. 55–67. \url{https://doi.org/10.1080/00401706.1970.10488634}
	
	\bibitem{app4} D. Hsu, Identifying key variables and interactions in statistical models of building energy consumption using regularization, \emph{Elsevier Energy}, \textbf{83} (2015), pp. 144–155. \url{https://doi.org/10.1016/j.energy.2015.02.008}
	
	\blue{\bibitem{KingmaBa} D. P. Kingma, J. Ba, Adam: A method for stochastic optimization, \textit{arXiv preprint arXiv:1412.6980} (2014). \url{https://doi.org/10.48550/arXiv.1412.6980}}
	
	\bibitem{app3} A. Khalajmehrabadi, N. Gatsis, D. J. Pack and D. Akopian, A Joint Indoor WLAN Localization and Outlier Detection Scheme Using LASSO and Elastic-Net Optimization Techniques, \emph{IEEE Transactions on Mobile Computing}, \textbf{16} (8) (2017), pp. 2079–2092. \url{https://doi.org/10.1109/TMC.2016.2616465}
	
	\bibitem{LiuWangLiu} H. Liu, T. Wang and Z. Liu, A nonmonotone accelerated proximal gradient method with variable stepsize strategy for nonsmooth and nonconvex minimization problems, \emph{Journal of Global Optimization}, \textbf{89} (4) (2024), pp. 863–897. \url{https://doi.org/10.1007/s10898-024-01366-4}
	
	\bibitem{Malitsky} Y. Malitsky and K. Mishchenko, Adaptive Gradient Descent without Descent, \emph{Proceedings of the 37th International Conference on Machine Learning (online): PMLR 119}, 2020. \url{https://doi.org/10.48550/arXiv.1910.09529}
	
	\bibitem{app8} R. Muthukrishnan and R. Rohini, LASSO: A feature selection technique in predictive modeling for machine learning, \emph{IEEE International Conference on Advances in Computer Applications}, Coimbatore, India, 2016, pp. 18–20. \url{https://doi.org/10.1109/ICACA.2016.7887916}
	
	\bibitem{Nesterov} Y. Nesterov, \emph{Lectures on Convex Optimization (2nd edition)}, Springer Optimization and Its Applications (137), Springer, Berlin, 2010. \url{https://doi.org/10.1007/978-3-319-91578-4}
	
	\bibitem{Nocedal} J. Nocedal and S. J. Wright, \emph{Numerical Optimization (2nd edition)}, Springer Series in Operations Research and Financial Engineering, Springer, New York, 2006. \url{https://doi.org/10.1007/978-0-387-40065-5}
	
	\bibitem{ParikhBoyd} N. Parikh and S. Boyd, Proximal Algorithms, \emph{Foundations and Trends in Optimization}, \textbf{1} (3) (2014), pp. 127–239. \url{http://dx.doi.org/10.1561/2400000003}
	
	\bibitem{app0} V. Roth, The generalized LASSO, \emph{IEEE Transactions on Neural Networks}, \textbf{15} (1) (2004), pp. 16–28. \url{https://doi.org/10.1109/TNN.2003.809398}
	
	\bibitem{Tibshirani1996} R. Tibshirani, Regression shrinkage and selection via the lasso, \emph{Journal of the Royal Statistical Society Series B: Statistical Methodology}, \textbf{58} (1) (1996), pp. 267–288. \url{https://doi.org/10.1111/j.2517-6161.1996.tb02080.x}
	
	\bibitem{Sarkar2015} T. Sarkar, \emph{Lecture notes on Advanced Machine Learning – Learning with Combinatorial Structure: Proximal Gradient, Atomic Norms}, Computer Science \& Artificial Intelligence Laboratory, Massachusetts Institute of Technology, accessed July 2024. \url{http://people.csail.mit.edu/stefje/fall15/notes_lecture13.pdf}
	
	\blue{\bibitem{Schmitt2019} T. Schmitt, \emph{House sale prices for King County, version 2.0} [data set] (2019), acc`essed November 2024. \url{https://www.openml.org/search?type=data&status=active&id=42092}}
	
	\bibitem{Wang2007} H. Wang, G. Li and G. Jiang, Robust Regression Shrinkage and Consistent Variable Selection Through the LAD-Lasso. \emph{Journal of Business \& Economic Statistics}, \textbf{25} (3) (2007), pp. 347–355. \url{https://doi.org/10.1198/073500106000000251}
	
	\bibitem{app2} F. Wen, L. Chu, P. Liu and R. C. Qiu, A Survey on Nonconvex Regularization-Based Sparse and Low-Rank Recovery in Signal Processing, Statistics, and Machine Learning, \emph{IEEE Access}, \textbf{6} (2018), pp. 69883–69906. \url{https://doi.org/10.1109/ACCESS.2018.2880454}
	
	\bibitem{app1} H. Zhang, J. Wang, Z. Sun, J. M. Zurada and N. R. Pal, Feature Selection for Neural Networks Using Group Lasso Regularization, \emph{IEEE Transactions on Knowledge and Data Engineering}, \textbf{32} (4) (2020), pp. 659–673. \url{https://doi.org/10.1109/TKDE.2019.2893266}
	
	\bibitem{Zou2005} H. Zou, T. Hastie, Regularization and Variable Selection Via the Elastic Net, \emph{Journal of the Royal Statistical Society Series B: Statistical Methodology}, \textbf{67} (2) (2005), pg. 301–320. \url{https://doi.org/10.1111/j.1467-9868.2005.00503.x}
	
	\bibitem{app6} H. Zou, T. Hastie and R. Tibshirani, Sparse principal component analysis, \emph{Journal of computational and graphical statistics}, \textbf{15} (2) (2006), pp. \url{265–286. https://doi.org/10.1198/106186006X113430}
	
	\bibitem{app7} H. Zou and L. Xue, A Selective Overview of Sparse Principal Component Analysis, \emph{Proceedings of the IEEE}, \textbf{106} (8) (2018), pp. 1311-1320. \url{https://doi.org/10.1109/JPROC.2018.2846588}
	
\end{thebibliography}
\end{document}